\documentclass[a4paper,10pt]{article}
\usepackage[latin1]{inputenc}
\usepackage[english]{babel}
\usepackage{amsmath,amsthm,amsfonts,amssymb,amstext} 
\usepackage{hyperref} 
\usepackage{latexsym} 
\usepackage{fancyhdr} 
\usepackage{listings}
\usepackage[final]{graphicx}
\usepackage{array}
\graphicspath{{./fig/}}
\usepackage{setspace}
\usepackage{appendix}
\usepackage{verbatim} 
\usepackage[usenames]{color} 
\usepackage{algorithm,algorithmic}
\usepackage{microtype} 
\usepackage{subfig}
\usepackage{multirow} 
\usepackage{cite}
\usepackage{sidecap}
 
\newtheorem{theorem}{Theorem}[section]

\newtheorem{lemma}[theorem]{Lemma}
\newtheorem{proposition}[theorem]{Proposition}
\newtheorem{corollary}[theorem]{Corollary}
\theoremstyle{definition}
\newtheorem{definition}[theorem]{Definition}
\theoremstyle{remark}

\newtheorem{example}[theorem]{Example}

\def\bbrkt#1\ebrkt{\left\{\begin{aligned}\begin{array}{ll}#1\end{array}\end{aligned}\right.}
\def\bbrktT#1\ebrktT{\left\{\begin{aligned}\begin{array}{llll}#1\end{array}\end{aligned}\right.}

\makeatletter
\let\S\@undefined
\makeatother

\DeclareMathOperator{\rank}{rank}
\DeclareMathOperator{\maxi}{max}
\DeclareMathOperator{\conv}{conv}
\DeclareMathOperator{\spaan}{span}

\newcommand{\D}{\Delta(\mathcal{L}_M - \{ \hat{0},\hat{1}\})}
\newcommand{\N}{\mathcal{N}(\mathcal{L}_M,G_{min})}
\newcommand{\Mw}{M_{\omega}}
\newcommand{\R}{\mathbb{R}}

\usepackage{tikz}
\usetikzlibrary{matrix,arrows}
\usetikzlibrary{arrows,shapes,automata,backgrounds,petri}
\tikzstyle{knoten}=[fill,shape=circle,inner sep=2pt,outer sep=0pt,minimum size=2pt]
\tikzstyle{background}=[rectangle,
                                                fill=gray!10,
                                                inner sep=0.2cm,
                                                rounded corners=5mm]
\usepackage{graphicx}

\title{New light on Bergman Complexes by decomposing matroid types}
\author{Martin Dlugosch}
\usepackage{accents}

\begin{document}
\maketitle
\begin{abstract}
Bergman complexes are polyhedral complexes associated to matroids. 
Faces of these complexes are certain matroids, called matroid types, too. 
In order to understand the structure of these faces we decompose matroid types into direct summands. 
Ardila/Klivans proved that the Bergman Complex of a matroid can be subdivided 
into the order complex of the proper part of its lattice of flats. 
Beyond that Feichtner/Sturmfels showed that the Bergman complex can even be subdivided to the even coarser nested set complex. 
We will give a much shorter and more general proof of this fact.
Generalizing formulas proposed by Ardila/Klivans and Feichtner/Sturmfels for special cases, 
we present a decomposition into direct sums working for faces of any of these complexes.  
Additionally we show that it is the finest possible decomposition for faces of the Bergman complex. 
\end{abstract}

\section{Introduction}
\label{sec1}
Let $V$ be a r-dimensional subspace of the n-dimensional vector space $\mathbb{C}^n$. 
The set of vectors 
\begin{displaymath}
 (\log |v_1|,\ldots,\log |v_n|)~ \in \R^n,  
\end{displaymath}
for $v_1,\ldots,v_n$ running through all non-zero elements of $\mathbb{C}^n$, is called the \emph{amoeba} of $V$. 
The limit set of these amoebas, for bases of the logarithm approaching zero, is a polyhedral fan called the \emph{Bergman fan of $V$}. 
It first appeared in the original paper of Bergman \cite{BE71} as \emph{logarithmic limit set of $V$}. 
The study of these spaces is stated as \emph{tropical geometry}. 

Matroid theory comes in when assigning a matroid to $V$ by setting its circuits $C$ as the minimal sets 
for which there are linear forms of the form 
$\sum_{i\in C} a_i x_i$ vanishing on $V$. 
For introductory references on matroid theory see \cite{Ox11}. 
In fact the Bergman fan of $V$ just depends on this associated matroid \cite{STU02}. 
It is the set of all vectors $\omega=(\omega_1,\ldots,\omega_n)\in \R^n$ such that for every circuit $C$ 
the minimum of the set $\{ \omega_i | i \in C\}$ is attained at least twice. 
Although we cut our own path of defining it, from this one can already see that 
the Bergman fan is invariant under translation along $\R(1,\ldots,1)$ and positive scaling. 
Hence we lose no information when restricting to the sphere 
$\mathbb{S}=\{ \omega \in R^n : \sum_{i=1}^n \omega_i=~0~,~\sum_{i=1}^n \omega_i^2 =1\}$. 
This restriction is a polyhedral complex called \emph{Bergman complex}. 

Section \ref{sec2} is devoted to the gathering of current knowledge about matroid polytopes and Bergman fans. 
Most of it is found in \cite{FS04} and \cite{AK}. 
Afterwards we take a look at the concept of nested set complexes by Feichtner and Kozlov \cite{FK04} in Section \ref{sec3}. 

The original result of Ardila and Klivans \cite{AK} is that the Bergman complex of a matroid $M$ can be subdivided to a realisation of the 
order complex of the proper part of the lattice of flats of $M$. 
The latter complex is well known \cite{edWhite}. 
This result was sharpened by Feichtner and Sturmfels \cite{FS04} by the fact that Bergman complexes can even be subdivided to realizations 
of the even coarser nested set complexes of their respective lattices of flats. 
Comparing faces of all these complexes by focusing on their vertices, 
we give a new, much shorter proof of the latter result in Section \ref{sec4}. 

Both Ardila/Klivans \cite[Prop. 2]{AK} and Feichtner/Sturmfels \\ \cite[Thm.~4.4]{FS04} gave formulas for the supporting matroid types 
of faces of the order complex respectively the nested set complex in terms of a decomposition into direct sums. 
In Section \ref{sec5} we generalize both these formulas giving such a decomposition that even works for faces of the Bergman complex. 
Taking a closer look we prove that the decomposition for faces of the Bergman complex is the finest one can get \textit{i.e.} 
summands are connected. 

In the end we give an interesting outlook to further research. 
It kind of all comes down to decomposition of intervals in a poset. 
The aim of Section \ref{sec6} is to describe methods how one could generalize the concepts used in this paper 
for defining a Bergman complex for arbitrary lattices or maybe even for arbitrary posets. 

\section{Matroid polytopes, the Bergman complex and nested sets}
\label{sec2}
We start with a geometric approach to matroids due to Gel'fand, Goresky, MacPhersen and Serganova \cite{GelGor}. 
Let $M$ be a family of subsets of size $r$ of a ground set $\{ a_1,\ldots,a_n \} $. 
Each subset $b$ can be represented by its incidence vector in $\R^n$, \textit{i.e.} the j-th coordinate is 1 
iff $a_j \in b$ and 0 otherwise. 
Now we can identify $M$ with the convex hull of its elements as incidence vectors:
\begin{displaymath} P_M := \conv \{e_b ~:~ b \in M    \}.\end{displaymath}
This yields a convex polytope in $\R^n$. Since the generating vertices all lie on the simplex 
$\Delta:=\left\{ (x_1, \ldots , x_n) \in \R^n ~:~0 \leq x_i \text{ for all } i, ~\sum_{i=1}^n x_i =r \right\} $, 
the dimension is limited by $n-1$. 
\begin{definition}
Let $M$ be a family of $r$-element subsets of the ground set $\{ a_1,\ldots,a_n \} $. 
If every edge of the polytope $P_M$ is parallel to 
$e_i -e_j$ for some $1\leq i,j \leq n$ we call $M$ a \emph{matroid}. The elements of $M$ are called bases of the matroid. The polytope $P_M$ is 
the \emph{matroid polytope}.
\end{definition}
Let $M$ be a matroid with ground set $E(M)$. A subset $I\subseteq E(M)$ is
\emph{independent} 
if $I\subseteq b$ for a basis $b \in M$. Hence, bases are the maximal independent sets. If not independent we call a subset of $E(M)$
\emph{dependent}. The \emph{rank of a matroid} is simply $r$, the cardinality of its bases. 
The \emph{rank of a subset} of $E(M)$ is the cardinality of its largest independent subset. Hence, a set is independent iff its 
cardinality equals its rank. 
With \emph{circuits} we denote minimal dependent sets with respect to inclusion. 
A \emph{flat} $F$ of $M$ is a subset such that there is no circuit $c$ of $M$ with $|c \backslash F|=1$. 
Another way to look at this is that 
every new element we add to $F$ is increasing the rank. 
The \emph{span} of a subset $G$ is the intersection of all flats containing $G$. 
It is the smallest flat which contains $G$.
We can construct it, by adding all the elements of $E(M) \backslash G$ to $G$ which do not increase the rank of $G$. 
The collection of flats of $M$ can be ordered by inclusion.
 The resulting poset is a lattice by setting $F_1 \wedge F_2 := F_1 \cap F_2$ and $F_1 \vee F_2 := \spaan(F_1 \cup F_2)$. We call it the 
geometric lattice $\mathcal{L}_M$ of the matroid $M$.  
A matroid $M$ is called \emph{loopless} if
$\bigcup\limits M = \bigcup\limits_{b \in M} b = E(M)$.

Let $F$ and $G$ be flats of $M$ such that $F$ is contained in $G$. We create a new matroid on the ground set $G \backslash F$ by setting:
\begin{displaymath}M[F,G]:=\{ b \cap (G\backslash F) ~:~ b\in M , |b\cap F| = \rank(F), |b\cap G|=\rank(G)\}.\end{displaymath}
The geometric lattice $\mathcal{L}_{M[F,G]}$ is isomorphic to the interval $[F,G]$ in $\mathcal{L}_M$. 

There are two special cases of this we want to bring up. The first is $F=\hat{0}$. We will call $M[\hat{0},G]$ the \emph{restriction} 
of $M$ to $G$.
Notice that the rank of the matroid $M[\emptyset,G]$ equals $\rank(G)$. 
The other special case is called \emph{contraction} and describes the case of $G=\hat{1}=E(M)$. 
Note that $M[F,E(M)]$ is loopless iff $F$ is a flat of $M$.


There is an equivalence relation on the 
ground set of a matroid defined as follows: Two elements $x$ and $y$ are equivalent if $x=y$ or there is a circuit of $M$ which contains both $x$ and $y$.
The proof that this is an equivalence relation is given in \cite[124-125]{Ox11}, like many other helpful statements for matroids. 
Note that connected matroids are loopless for ground sets of cardinality of at least 2. 
The equivalence classes are the \emph{connected components} of $M$. 
Let $c(M)$ denote the number of connected components of $M$. We say a matroid is \emph{connected} iff $c(M)\leq 1$. 
The case of $c(M)=0$ belongs solely to the empty matroid which is the matroid with empty ground set.
 
We say a flat $F$ of $M$ is \emph{connected} if its restriction $M[\hat{0},F]$ is connected. Dual to this, it is called 
\emph{co-connected} if its contraction $M[F,\hat{1}]$ is connected. 

\begin{proposition}
 The dimension of the matroid polytope $P_M$ is $n-c(M)$. 
\end{proposition}
\begin{proof}
The linear space parallel to the affine space spanned by the matroid polytope is generated by vectors of the form $e_i-e_j$. An edge of $P_M$ 
is parallel to such a difference of unit vectors iff the vertices at the end of that edge represent bases $b,b'$ just differing in 
the two elements $a_i$ and $a_j$. But this is in fact an equivalent condition to the existence of a circuit containing both. 
So $\text{dim}(P_M) \leq n-c(M)$.

Conversely, 
every pair of connected elements $a_i,a_j$ grants the existence of two bases $b,b'$ differing just in these elements. So there is an edge of $P_M$
parallel to $e_i-e_j$. Thus $\text{dim}(P_M) \geq n-c(M)$.
\end{proof}

From here on we are focusing on connected matroids. Every non-connected matroid can be decomposed into a direct sum of connected matroids. 
Then the matroid polytope of the original matroid is just the product of the matroid polytopes of the direct summands.

\newpage
Our goal now is to describe the matroid polytope 
by a system of linear (in-)equalities.
\begin{proposition}
 For a connected matroid $M $ of rank $r$ the matroid polytope has the form:
\begin{displaymath} P_M = \left\{  (x_1,\ldots , x_n) \in \Delta ~:~ \sum_{i \in F} x_i \leq \rank(F) \text{ for all flats $F$ of $M$} \right\}. 
\end{displaymath}
\end{proposition}
\begin{proof}
 It is enough to consider the facets of $P_M$, since they are bounding the polytope. 
Let us assume a facet defining inequality is $\sum_{i=1}^n a_i x_i \leq c$, where $n$ is the cardinality of $E(M)$. What we know from 
the definition of matroids is that all the edges are parallel to vectors $e_i -e_j$. We can compute the normal vector of the facet by 
looking at the edges it has to be perpendicular to. The constraints the edges impose are all of the form $a_i=a_j$. Since the normal vector is 
uniquely determined by these constraints up to scalar multiples we can think of the $a_i$ to be either 0 or 1. So the inequality reduces to 
$\sum_{i \in A} x_i \leq c'$ for some $A \subseteq E(M)$ and some $c'\in \R$. But what is the maximum the linear form $\sum_{i\in A} x_i$ 
can attain? As a linear form its maximum has to be attained at some face of $P_M$. Pick any vertex of this face and evaluate the 
linear form on it. We obtain that $c'=\maxi\ \{ |b \cap A |~:~ b \in M \} =\rank(A)$. 
What is left is to show that we just have to pick the subsets $A$  which are flats. Let $\spaan(A)$ be the flat spanned by $A$. Of course 
$A \subseteq \spaan (A) $ and $\rank(A)$=$\rank(\spaan(A))$ hold.
\begin{displaymath} \sum_{i \in A} x_i \leq \sum_{i \in \spaan(A)} x_i \leq \rank ( \spaan(A) ) = \rank(A) \end{displaymath} 
So the inequality of $\spaan(A)$ in the middle already implies the inequality of~$A$. 
\end{proof}
A very important fact, still easy to see, is that every face of $P_M$ is a matroid polytope, too, because its edges are still parallel to some
difference of standard basis vectors. Assume that, for $\omega \in \R^n$, the linear form 
$\sum_{i =1}^n \omega_i x_i $ attains its maximum in $P_M$ at that chosen face.
Notice that this linear form can differ from the linear forms in the proposition above. 
It is not uniquely determined by the chosen face. Considering this we will construct the Bergman complex soon.  
The bases of this face are exactly the bases of $M$ for which the linear form 
$\sum_{i =1}^n \omega_i x_i $ attains its maximum. 
From an algorithmic point of view the bases are exactly the possible outputs of the greedy algorithm with weight $\omega$. 
The matroid $\Mw$ is called the \emph{matroid type} of the face.

Since we can represent the matroid polytope $P_M$ through a system of linear inequalities indexed by the flats of $M$
we want to filter which of them define facets of the matroid polytope. 
The question is what 
combinatorial property do flats have whose linear form attains its maximum at a facet of the matroid polytope. 
We will call these special flats \emph{flacets} of $M$. 

So the matroid polytope $P_M$ is bounded by the inequalities of the flacets and 
possibly the inequalities of the form $x_i \geq 0$. 
Especially matroids are determined by their flacets. 

\newpage
\label{maxfaces}
\begin{proposition}
 A flat $F$ of $M$ is a flacet iff it is connected and co-connected. 
\end{proposition}
\begin{proof}
We already saw that for connected matroids the matroid polytope has dimension $n-1$. So facets of this pure polyhedral complex have dimension $n-2$. 
We just have to look out for those flats $F$ such that the matroid polytope of the matroid type $M_{e_F}$ has dimension $n-2$.
We know what the bases of $M_{e_F}$ are. They are exactly the bases of $M$ such that $|b \cap F| = \rank(F)$ holds. We can express this 
matroid type through the constructions of restriction, contraction and direct sum: 
\begin{displaymath} M_{e_F} = M[\emptyset,F] \oplus M[F,E(M)] = \{ b\in M : |b \cap F|= \rank(F) \} \subseteq M .\end{displaymath}
Since this matroid type is a direct sum, its matroid polytope, which is a facet of $P_M$, is a product of the matroid polytopes of 
$M[\emptyset,F]$ and $M[F,E(M)]$. So~the dimension of the face defined by $F$ must equal the sum of the dimensions of the matroid polytopes 
which is 
\begin{displaymath}|F|- c(M[\emptyset,F]) + n-|F|-c(M[F,E(M)])= n - c(M[\emptyset,F]) -c(M[F,E(M)]).\end{displaymath}
From this we can see that the dimension $n-2$ appears iff  \begin{displaymath}c(M[\emptyset,F]) + c(M[F,E(M)]) = 2.\end{displaymath} 

So if $F$ is neither the empty nor the whole set both summands have to equal 1 
and this means the flacets are exactly those flats for which both restriction and contraction are non-empty 
connected matroids. 
\end{proof}

\bigskip
\label{firstdecomp}
\noindent
The idea of decomposing the matroid type $\Mw$ in as fine of a direct sum as possible is the aim of Section \ref{decompositionchapter}. 
The case of $M_{e_F}$ above is a first example of this. 

Recall that the linear form $\sum_{i=1}^n \omega_i x_i $ 
attaining its maximum at a certain face of $P_M$ is not uniquely determined. 
For $c \in \R $ and $c^+ \in \R^+$  the linear forms 
$\sum_{i=1}^n (\omega_i \cdot c^+) x_i $ and $\sum_{i=1}^n (\omega_i+c) x_i $ induce the same matroid type. 
So we still get all the different matroid types when restricting $\omega$ to elements of the unit ($n-2$)-sphere contained in 
a hyperplane orthogonal to $(1,\ldots,1)$, which is 
$\mathbb{S}=\{ \omega \in R^n : \sum_{i=1}^n \omega_i=~0,\sum_{i=1}^n \omega_i^2 =1\}$. 
The only exception is the matroid type of $\omega = (0,\ldots,0)$ 
which is simply $M$ again. 
Later we will view this as the matroid type of the empty face of the Bergman complex.

Consider the following equivalence relation on $S$:~$\omega \sim \omega' $ iff the induced matroid types $\Mw$ and $M_{\omega'}$ coincide. 
The equivalence classes are relatively open convex polyhedral cones. These cones form a complete fan in $\R^n$. It is the normal fan of $P_M$. 
The equivalence classes define a spherical subdivision of $\mathbb{S}$. 
This subdivision is isomorphic to the boundary of the polar dual $P_M^*$ of the 
matroid polytope. 
For simplicity of notation, we will identify the face of $\partial P_M$ just like its dual in $\partial P_M^*$ with its 
matroid type $\Mw$. 
\begin{definition}
The \emph{Bergman Fan} $\widetilde{B}(M)$ is the subfan of the inner fan of $P_M$ consisting of the matroid types which are loopless. 
The \emph{Bergman Complex} $B(M)$ is the intersection  $\widetilde{B}(M)\cap \mathbb{S}$ where 
$\mathbb{S}=\{ \omega \in \R^n : \sum_{i=1}^n \omega_i= 0~,~\sum_{i=1}^n \omega_i^2 =1\}$. 
\end{definition}

After this geometric realisation let us reduce the Bergman complex to its combinatorial data, \textit{i.e.} to its face poset.
The faces of the Bergman complex $B(M)$ are the matroid types $\Mw$ which are loopless. Since the matroid types are subsets of bases 
of $M$ they come with the natural partial order of inclusion. 
The order of the face poset of the Bergman complex is the dual of this order. So the face $\Mw$ is contained in the face 
$M_{\omega'}$ in $B(M)$ iff the reversed inclusion holds for $\Mw$ and $M_{\omega'}$ as subsets of $M$. 

\bigskip 
Let us consider what the vertices of the Bergman complex are. Apart from the empty face 
they are the minimal faces of $P_M^*$ \textit{i.e.} the maximal faces of $P_M$, whose matroid types are loopless. 
But obviously the matroid type of a face of $P_M$ is loopless iff the face is not contained in one of the hyperplanes of the form $x_i=0$. 
So the vertices of $B(M)$ are duals of facets of $P_M$ whose hyperplanes are not of the form $x_i=0$. 
We already determined what these are, namely the flacets of $M$. 
Since every face of $P_M$ is uniquely determined by the set of facets of the polyhedral complex which contain it, 
every face of $B(M)$ is uniquely determined by its vertices.
             
\begin{example}
\label{ex2}
 Let $M$ be the matroid with ground set $\{ 1,2,3,4,5,6\}$ and circuits $ \{1,2,3,4 \},\{1,2,5,6 \}$ and $\{ 3,4,5,6\}$. 
Hence the bases are all subsets of size four except for the circuits. 
There are two types of flacets here. 
On the one hand there are the singletons \textit{i.e.} the subsets of size one. 
On the other hand there are the three circuits itself. 

The Bergman complex is a pure polyhedral complex of dimension two. 
There are two kinds of facets of the Bergman complex. 
There are twenty triangles but there are also three quadrangles, whose vertices are shortly notated: 
\begin{displaymath}
 1,2,1234,1256 ~\quad~ 3,4,1234,3456 ~\quad~ 5,6,1256,3456 .
\end{displaymath}
This shows that the Bergman complex can be non-simplicial. 
\end{example}
                                                                                                    
\section{Nested set complexes}
\label{sec3}

Beside order complexes this is another way of creating a simplicial complex from a geometric lattice 
due to Feichtner and Kozlov \cite{FK04}. 

For a semi-meet lattice $\mathcal{L}$ let intervals in $\mathcal{L}$ be denoted by 
$[X,Y]:= \{Z \in~\mathcal{L}:  X~\leq~Z~\leq Y \}$. For any $X \in \mathcal{L}$ and any subset 
$S\subseteq \mathcal{L}$ write 
$S_{\leq X}:= \{ Y \in S : Y \leq X\}$. The same way we can define $S_{< X},S_{\geq X} $ and $S_{> X}$. Last but not least, the set of
maximal elements in $S\subseteq \mathcal{L}$ is denoted by $\maxi S$.
\begin{definition}
 For a finite lattice $\mathcal{L}$ a subset $\mathcal{G}$ in $\mathcal{L}_{> \hat{0}}$ is a \emph{building set} 
if for any $X\in \mathcal{L}_{> \hat{0}}$ with 
$\maxi \mathcal{G}_{\leq X}=\{ G_1,\ldots ,G_k \}$ the map 

\begin{displaymath} \Phi_X : \Pi_{j=1}^k [\hat{0} , G_j]  \longrightarrow [\hat{0},X]\end{displaymath}
induced by the inclusions of the intervals $[\hat{0} , G_j] \subseteq [\hat{0} , X]$ is an isomorphism of posets. 
\end{definition}
In colorful language, the condition means that $X$ can be decomposed into $G_1,\ldots , G_k$ and the properties of $[\hat{0} , X]$ can be 
separately investigated in the intervals $[\hat{0} , G_j]$.   

There is always a maximal and a minimal building set. The maximal one is always the whole lattice without $\hat{0}$. In this case $X$ is 
decomposed into just one factor, $X$ itself. The minimal building set $\mathcal{G}_{min}$ 
consists of all connected flats and, if $\mathcal{L}$ is not connected already, the top element $\hat{1}$. 
\begin{definition}
 Let $\mathcal{L}$ be a finite lattice and $\mathcal{G}$ a building set containing the top element $\hat{1}$. A subset $S\subseteq \mathcal{G}$ 
is called nested if for any set of 
incomparable elements $X_1,\ldots ,X_k$ of at least two elements of $S$ 
the join $X_1 \vee \ldots \vee X_k$ does not lie in the building set $\mathcal{G}$ again. Since subsets of $S$ 
fulfill the same condition, again this is a simplicial complex. Topologically, it is a cone with apex $\{ \hat{1} \} $. Its link 
$\mathcal{N}(\mathcal{L},\mathcal{G})$ is the \emph{nested set complex} of $\mathcal{L}$ with respect to the building set $\mathcal{G}$.
\end{definition}
\label{topbuildingset}
The case of the minimal building set $\mathcal{G}_{min}$ is called \emph{the} nested set complex 
$\mathcal{N}(\mathcal{L},\mathcal{G}_{min})$. If there is no hint 
about the building set, it is the minimal one. 
The other extreme is the maximal building set. Since every join of incomparable elements is an element of the building set the only
nested sets are those which are totally ordered. So the simplices are just the chains in the proper part $\mathcal{L}-\{\hat{0},\hat{1} \}$.
Thus the nested set complex with respect to the maximal building set equals the order complex of the proper part of $\mathcal{L}$.

\bigskip
A lattice is \emph{atomic} if every element is a join of atoms. Our geometric lattices are such atomic lattices. 
For simple matroids these atoms are just the elements of the ground set $E(M)$. 
For arbitrary atomic lattices Feichtner and Yuzvinsky \cite{FY04} proposed the following polyhedral realization of nested set complexes.

Let $\mathcal{L}$ be an atomic lattice with atoms $\{ a_1,\dots,a_n \}$ and $\mathcal{G}$ a building set containing $\hat{1}$. 
For any $G \in \mathcal{G}$ let $e_G \in \R^n$ be the incidence vector of $G$ respective the set of atoms 
\textit{i.e.} the i-th coordinate is 1 iff $a_1 \subseteq G$ and 0 otherwise. 
For a nested set $S$ the set of incidence vectors of its elements is linearly independent. 
Hence with $\R_{\geq 0} \{ e_G | G\in S \}$, they span a simplicial cone. 
For nested sets $S$ and $S' $ they intersect exactly in the cone belonging to the nested set $S \cap S'$. 
Thus the set of cones from nested sets form a simplicial fan. 

Just like the Bergman Fan this fan has the property that its cones are invariant under the translation along the line $\R(1,\ldots,1)$. 
So again we loose no information when restricting the fan to the (n-2)-sphere 
$\mathbb{S}=\{ \omega \in R^n : \sum_{i=1}^n \omega_i=~0~,~\sum_{i=1}^n \omega_i^2 =1\}$. 
The resulting spherical complex is a geometric realization of the nested set complex.
    
In order to compare nested set complexes of different building sets, Feichtner and Müller \cite{FM05} proved that for building sets 
$\mathcal{G}$ and $\mathcal{G}\cup \{ X \}$ the nested set complex respective the building set $\mathcal{G}\cup \{ x \}$ can be 
obtained by a stellar subdivision from the nested set complex of the smaller building set $\mathcal{G}$ at the simplex corresponding to the factors 
of $X$ respective $\mathcal{G}$. 
Recursively we can construct the order complex of the proper part $ \mathcal{L}- \{ \hat{0} , \hat{1} \} $ 
, which is the nested set complex of the maximal building set,  
from the nested set complex of the minimal building set by a sequence of stellar subdivisions. 
The single steps correspond to adding elements to the building sets in a non decreasing order. 
In particular the order complex of the proper part of $\mathcal{L}$ and \emph{the} nested set complex are homeomorphic. 

\newpage
\begin{example}
\label{ex3}
Consider the matroid of Example \ref{ex2}. 
The lattice $\mathcal{L}$ is the lattice of flats of the matroid $M$. 
Though this is not true in general, the minimal building set $\mathcal{G}_{min}$ consists exactly of the set of flacets. 
The complexes coincide except for the three squares which are each replaced by two triangles with vertices:  
\begin{displaymath}
 1,2,1234~\quad~ 1,2,1256 ~\quad~ 3,4,1234 ~ \quad~ 3,4,3456 ~\quad~ 5,6,1256~\quad~ 5,6,3456 .
\end{displaymath}
\end{example}

\section{Comparison of order complex, nested set complex and Bergman complex}
\label{sec4}
Ardila and Klivans \cite{AK} first showed that the order complex of proper part of the lattice of flats is a refinement of the Bergman complex.  
Feichtner and Sturmfels \cite{FS04} proved that this is even true for the nested set complex. 
Though we can gain a lot of insight from their proof it is a little bit complicated. 
From another point of view this can be seen easier.

\bigskip

Remember the geometric realizations of 
the order complex (nested set complex of the maximal building set), 
the nested set complex (with minimal building set) 
and the Bergman complex.

What faces of all these complexes have in common is that they are the spherical convex of their vertices. 
The vertices are the scaled incidence vectors of flats in the case of the order complex, 
connected flats in the case of the nested set complex 
and flacets in the case of the Bergman complex.  

\begin{theorem}
\label{theo1}
For any of these complexes any element $\omega$ in the cone corresponding to some face with vertices $\Gamma$ has the form 
$\sum_{F\in \Gamma} \lambda_F \cdot e_F $ for all $\lambda_F \gneq 0$. Then identifying the matroid type $\Mw$ with its set of bases:
\begin{displaymath}M_{\omega}=\left\{~ b \in M ~|~ \text{ for all }F \in \Gamma ~:~ |b \cap F|= \rank(F)\right\}.\end{displaymath}
\end{theorem}
\begin{proof}
A basis $b$ of $ M$ maximizes the linear functional $\omega$ in $P_M$ 
iff for all bases $b'$ of $M$ the inequality $e_{b'} \cdot \omega \leq e_b \cdot \omega$ holds. 
First consider a basis $b$ satisfying $|b \cap F|= \rank(F)$ for all $F \in \Gamma$.            
 \begin{align*}
 e_{b'} \cdot \omega 
&=e_{b'} \cdot \left(\sum_{F \in \Gamma} e_F\right)
=\sum_{F \in \Gamma} e_{b'} \cdot e_F 
=\sum_{F \in \Gamma} |b' \cap F| \\
&\leq \sum_{F \in \Gamma} \rank(F) 
= \sum_{F \in \Gamma} |b \cap F|
= \sum_{F \in \Gamma} e_b \cdot e_F 
= e_b \cdot \left(\sum_{F \in \Gamma} e_F\right)
= e_b \cdot \omega 
\end{align*}
This shows that $b$ is at least a basis of $\Mw$. 

Conversely assume $e_{b'} \cdot \omega \leq e_b \cdot \omega$ always holds. 
Choose $b'$ such that $|b~\cap~F|= \rank(F)$ holds for all $F \in \Gamma$. 
Then the equations above teaches us that $\sum_{F \in \Gamma} |b~\cap~F| = \sum_{F \in \Gamma} \rank(F)$. 
Since for all pairs of summands the inequality $|b \cap F| \leq  \rank(F)$ holds, equality holds for them, too.  
\end{proof}

\newpage
\begin{corollary}
\label{theo2}
For all building sets $\mathcal{G}$ the nested set complex $N(\mathcal{L}_M,\mathcal{G})$ is a refinement of the Bergman fan. 
\end{corollary}
\begin{proof}
Already Ardila and Klivans \cite{AK} showed that the $\omega \in \R^n$ for which $\Mw$ is loopless are exactly the ones 
lying in the interior of polyhedral cones spanned by incidence vectors of flats. 
Due to Theorem \ref{theo1} the induced matroid types are the same for all elements of any face of the realizations of our complexes. 
\end{proof}

\begin{example}
\label{ex4}
Consider the matroid of Example \ref{ex2} and \ref{ex3}. 
In Figure \ref{fig:composition2} we compare the excerpts of the order complex $\Delta(\mathcal{L}-\{\hat{0},\hat{1} \})$, 
the nested set complex $\N$ and the Bergman complex $B(M)$ which arise as subdivisions of each other.  
The matroid type of the two-dimensional face is $\Mw~=~\{ 1235,1236,\\ 1245,1246 \}$. 
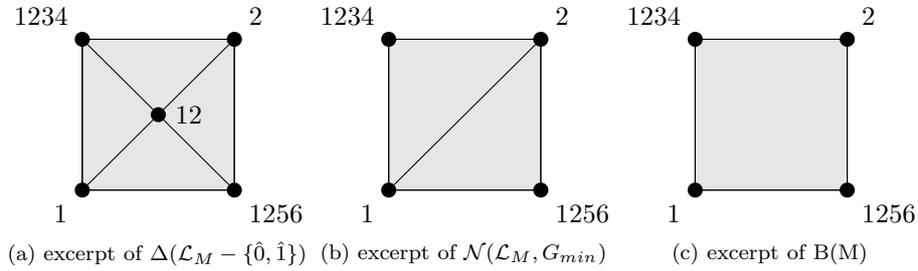
\begin{figure}[h]
  \centering
  \subfloat[excerpt of $\D$]{\begin{tikzpicture}[description/.style={fill=white,inner sep=2pt}]
\filldraw[fill=gray!20] (0,0) -- (0,2) -- (2,2) -- (2,0) -- (0,0);\
\node (1234) at (0,2)[knoten,label=above left:1234]{};
\node (2) at (2,2)[knoten,label=above right:2]{};
\node (1) at (0,0)[knoten,label=below left:1]{};
\node (12) at (1,1)[knoten,node distance=2.12cm,label=right:12]{};
\node (1256) at (2,0)[knoten,label=below right:1256]{};
\path[-] (1)  edge (1234);
\path[-] (1)  edge (1256);
\path[-] (2)  edge (1234);
\path[-] (2)  edge (1256);
\path[-] (1)  edge (12);
\path[-] (12)  edge (2);
\path[-] (12)  edge (1256);
\path[-] (12)  edge (1234);
\end{tikzpicture}}                
  \subfloat[excerpt of $\N$]{\begin{tikzpicture}[description/.style={fill=white,inner sep=2pt},node distance=3cm]
\filldraw[fill=gray!20] (0,0) -- (0,2) -- (2,2) -- (2,0) -- (0,0);\
\node (1234) at (0,2)[knoten,label=above left:1234]{};
\node (2) at (2,2)[knoten,label=above right:2]{};
\node (1) at (0,0)[knoten,label=below left:1]{};
\node (1256) at (2,0)[knoten,label=below right:1256]{};
\path[-] (1)  edge (1234);
\path[-] (1)  edge (1256);
\path[-] (2)  edge (1234);
\path[-] (2)  edge (1256);
\path[-] (1)  edge (2);
\end{tikzpicture}}
 \subfloat[excerpt of B(M)]{\begin{tikzpicture}[description/.style={fill=white,inner sep=2pt},node distance=3cm]
\filldraw[fill=gray!20] (0,0) -- (0,2) -- (2,2) -- (2,0) -- (0,0);\
\node (1234) at (0,2)[knoten,label=above left:1234]{};
\node (2) at (2,2)[knoten,label=above right:2]{};
\node (1) at (0,0)[knoten,label=below left:1]{};
\node (1256) at (2,0)[knoten,label=below right:1256]{};
\path[-] (1)  edge (1234);
\path[-] (1)  edge (1256);
\path[-] (2)  edge (1234);
\path[-] (2)  edge (1256);
\end{tikzpicture}}
\caption{The same excerpt of all polyhedral complexes}
\label{fig:composition2}
\end{figure}
\end{example}

\section{Decomposition of matroid types}
\label{sec5}
\label{decompositionchapter}

\begin{definition}
 Let $A$ be a subset of $E(M)$. We say that $A$ has \emph{full $\omega$-rank} if for all bases
$b \in M_{\omega} $, $ |A \cap b |= \rank(A)$ holds.   
\end{definition}
\noindent Note that the vertices of a face have full $\omega$-rank by Theorem \ref{theo1}. 
Here is an immediate application of this important definition. 
\begin{proposition}
\label{prop1}
A matroid type $\Mw$ is loopless iff all sets $A$ with full $\omega$-rank are flats of $M$. 
\end{proposition}
\begin{proof}
 Assume $A$ is not a flat, then there exists a circuit $c$ of $M$ such that $c \setminus A = \{ x \}$. For all bases $b \in M_{\omega}$: 
\begin{displaymath} |A \cap b  | = \rank(A) = \rank(A \cup \{ x \}) \geq |(A \cup \{ x \}) \cap b  |. \end{displaymath}
So $b$ can not contain $x$. Since this holds for all bases of $M_{\omega}$, this means $x$ is a loop of $\Mw$. 
The same argument works for the other implication, too. 
\end{proof}
\newpage
\begin{proposition}
\label{prop2}
 Let $A$ and $B$ both having full $\omega$-rank. Then $A \cap B$, $A \cup B$ and 
every connected component of a set with full $\omega$-rank have full $\omega$-rank, too.
\end{proposition}
\begin{proof}
For all bases $b \in \Mw$: 
\begin{align*}
	\rank(A \cup B) 
 \geq &	  |  ( A \cup B) \cap b | \\
=	 & | A \cap  b  |   +  |  B \cap  b |- |  (A \cap B)  \cap  b  |)\\
= 	 & \rank(A) + \rank(B) - |  ( A \cap B) \cap b |.
\end{align*}

\noindent With this inequality and the submodularity of the rank function in mind it follows:
\begin{align*}
\rank(A) + \rank(B) 
~ & \leq ~
 \rank(A \cup B) + |  ( A \cap B) \cap b | \\ 
~ & \leq ~ 
\rank(A \cup B) + \rank(A \cap B)
~ \leq ~
\rank(A) + \rank(B).
\end{align*}
From this we conclude that $|  ( A \cap B) \cap b |=\rank(A \cap B)$. So $A \cap B$ has full $\omega$-rank, too.

\bigskip
\noindent Now we want to show the same for $A \cup B$:
\begin{align*}
	\rank(A \cup B) \geq  &~ |  ( A \cup B) \cap b | \\
=	 &~ |  A  \cap  b  | - |  (A \cap B)  \cap  b  |  +  
|  B  \cap  b  |\\
= 	 & ~\rank(A) + \rank(B) - \rank(A \cap B).
\end{align*}
Again, with this inequality and the submodularity of the rank function in mind, we get
\begin{displaymath}\rank(A) + \rank(B) \leq
\rank(A \cup B) + \rank(A \cap B)
 \leq 
\rank(A) + \rank(B).\end{displaymath}

\noindent The same way we conclude that $| ( A \cup B) \cap b | = \rank(A \cup B$). So $A \cup B$ has full $\omega$-rank, too. 

\bigskip
\noindent Let $A_1,\ldots,A_t$ denote the connected components of a set $A$ with full $\omega$-rank. For $b \in M_{\omega}$:
\begin{displaymath} \sum_{i=1}^t |b \cap A_i| = |b \cap A|= \rank(A)=\sum_{i=1}^t \rank(A_i). \end{displaymath}
Together with the pairwise inequality $|b \cap A_i| \leq \rank(A_i)$ we can conclude that $|b~\cap~A_i|~=~\rank(A_i)$. 
So all the $A_i$ have full $\omega$-rank, too.  
\end{proof}
With this knowledge we see that the set of subsets with full $\omega$-rank is a sublattice of $\mathcal{L}_M$. 
It has the property that the join is already the union 
instead of just its span. 
Additionally it is closed under taking connected components \textit{i.e.} if an interval $[A,B]$ in $\mathcal{L}_M$ is isomorphic to 
$[A,C_1]\times [A,C_2]$ for $A<C_1,C_2<B$ and $A,B$ have $\omega$-full rank then both $C_1$ and $C_2$ have this property, too.  
All information about $\Mw$ is contained in this special sublattice. 
It is the sublattice, closed under taking connected components,  
which is induced by the set of flacets and the elements $\hat{0},\hat{1}$. 
This observation is the starting point for section \ref{ausblick}.  

\begin{theorem}
\label{maintheo1}
Let $\Mw$ be the matroid type of any face of either the order complex $\D$, the nested set complex $\N$ or the 
Bergman complex $B(M)$ and $\Gamma$ its set of vertices.  
Additionally let $\Pi:=\bigvee_{F\in\Gamma}(F~ |~ E(M)-F)$ be the partition of $E(M)$ which is the join in $\Pi_{E(M)}$ 
of the partitions consisting of just two blocks, the vertex and its complement. 
For any block $\alpha$ in $\Pi$, let $\Gamma_{\alpha}$ denote the elements of $\Gamma$ which contain $\alpha$ as a subset. 
Then, 
 \begin{displaymath} M_{\omega} \cong \bigoplus\limits_{\alpha \in \Pi} M\left[~\bigcap \Gamma_{\alpha} ~ 
\backslash ~ \alpha,\bigcap \Gamma_{\alpha}~\right].\end{displaymath}
\end{theorem}

\begin{example}
\label{ex5}
 For the matroid $M$ of Example \ref{ex2} and the face of the Bergman complex with vertices $1,2,1234,1256$, the partition $\Pi$ is 
$1|2|34|56$. Its matroid type decomposition is 
\begin{displaymath}
 M[\emptyset,1]\oplus M[\emptyset,2]\oplus M[12,1234]\oplus M[12,1256].
\end{displaymath}
\end{example}

\begin{proof}
First of all notice that the ground set of the summand corresponding to $\alpha$ is exactly $\alpha$ 
so these ground sets form a disjoint union of the ground set $E(M)$. 

Notice that by construction of $\Pi$ a block $\alpha$ is either contained in or disjoint to a flacet $F$ of $\Gamma$. 
So, denoting the elements of $\Gamma$ disjoint to $\alpha$ as $\Gamma_{\alpha}^C$, 
the following equation exclusively follows from set theory: 
\begin{displaymath} \bigcap \Gamma_{\alpha} ~\cap~ \bigcup\limits \Gamma_{\alpha}^C = \bigcap \Gamma_{\alpha} - \alpha .\end{displaymath}

\noindent Due to Proposition \ref{prop2} the interval borders of the direct summands have full $\omega$-rank and 
hence they are flats by Proposition \ref{prop1}. 

\bigskip
For showing that the two matroids are equal, we will show that every basis of one of them is a basis of the other matroid, too.
\label{closelookmaintheo}

So for the first inclusion of bases let $b$ be any basis of the matroid $M_{\omega}$. The elements of $\Gamma$ have full $\omega$-rank 
by definition. By Proposition \ref{prop2} an intersection of subsets with full $\omega$-rank has full $\omega$-rank, too. 
This yields $|b \cap \bigcap \Gamma_{\alpha}|= 
\rank(\bigcap \Gamma_{\alpha})$. So $ b \cap \bigcap \Gamma_{\alpha}$ is a basis of $M[\emptyset,\bigcap \Gamma_{\alpha}]$.
By Proposition \ref{prop2} even $\bigcap \Gamma_{\alpha} ~\cap~ \bigcup \Gamma_{\alpha}^C$ has full $\omega$-rank. So 
$|b \cap (\bigcap \Gamma_{\alpha} ~\cap~ \bigcup \Gamma_{\alpha}^C)|= \rank(\bigcap \Gamma_{\alpha} ~\cap~ \bigcup \Gamma_{\alpha}^C)$ 
holds and 
$(b \cap \bigcap \Gamma_{\alpha} ~\cap~ \bigcup \Gamma_{\alpha}^C)$ is a basis 
of $M[\emptyset,\bigcap \Gamma_{\alpha} ~\cap~ \bigcup \Gamma_{\alpha}^C]$. 
Therefore $b \cap \alpha$ is a basis of the direct summand 
$M\left[~\bigcap \Gamma_{\alpha} ~\cap~ \bigcup \Gamma_{\alpha}^C,\bigcap \Gamma_{\alpha}~\right]$ and
$b = \bigcup_{\alpha \in \Pi} b \cap \alpha$ is a basis of 
$\bigoplus_{\alpha \in \Pi} M[~\bigcap \Gamma_{\alpha} ~\cap~ \bigcup \Gamma_{\alpha}^C,\bigcap \Gamma_{\alpha}~]$.

\bigskip
\noindent
For the second inclusion we have to show: Any basis 
$b'$ of $\bigoplus_{\alpha \in \Pi} M[\bigcap \Gamma_{\alpha} \backslash \alpha,\bigcap \Gamma_{\alpha}]$ 
is a basis of $M$, too. 
Afterwards we will show that the equation $|b' \cap F|=\rank(F)$ holds for all $F \in \Gamma$. 

\bigskip
\noindent Let $b$ be a basis of $M_{\omega}$, then  
\begin{align*}
\rank(M)=&|b|= \sum_{\alpha \in \Pi} |b\cap \alpha| \\
=& \sum_{\alpha \in \Pi} (|b \cap \bigcap \Gamma_{\alpha}| - |b \cap (\bigcap \Gamma_{\alpha} - \alpha)|) \\
=& \sum_{\alpha \in \Pi} (\rank(M[\emptyset,\bigcap \Gamma_{\alpha}]) - \rank(M[\emptyset,\bigcap \Gamma_{\alpha} -\alpha])) \\
=& \sum_{\alpha \in \Pi} \rank(M[\bigcap \Gamma_{\alpha} -\alpha ,\bigcap \Gamma_{\alpha}]) \\
=& \sum_{\alpha \in \Pi} |b' \cap \alpha| =|b'|.
\end{align*}
So we know the rank of $M[\bigcap \Gamma_{\alpha} \backslash \alpha,\bigcap \Gamma_{\alpha}]$ equals the rank of $M$ if there exists at least one 
basis $b$ fulfilling the conditions $|b \cap F|=\rank(F)$ for all $F\in \Gamma$. 
Now we want to prove that $b'$ has full rank in terms of the matroid $M$. Assume this is not the case. 
Then choose a minimal block $\alpha$ for which there exists $x\in \alpha$ such that by adding this to $b'$ we increase the rank 
in terms of the matroid $M$. 
Minimality is meant to be with respect to the order relation $\alpha \leq \beta$ iff $\bigcup \Gamma_{\alpha} \subseteq \bigcup \Gamma_{\beta}$.
\label{pagerefmaintheo}
\noindent Now, $\rank( \bigcap\Gamma_{\alpha} ) = \rank( \bigcap\Gamma_{\alpha} - \alpha ) +
 \rank( M[ \bigcap\Gamma_{\alpha} - \alpha, \bigcap\Gamma_{\alpha} ] $) holds. Since $b' \cap \bigcap\Gamma_{\alpha} - \alpha $ has full rank
in $M[\emptyset, \bigcap\Gamma_{\alpha} - \alpha $] by minimality and $b'\cap \bigcap\Gamma_{\alpha}$ has full rank in 
$M[ \bigcap\Gamma_{\alpha} - \alpha, \bigcap\Gamma_{\alpha} ]$ by construction of $b'$, we conclude that 
$b' \cap \bigcap \Gamma_{\alpha}$ has full rank in $M[ \emptyset , \bigcap \Gamma_{\alpha} ]$, too. 
But this is a contradiction to the choice of $\alpha$. 
So $b'$ has full rank in $M$ and it has the size $\rank(M)$. Therefore it is a basis of~$M$. 

\bigskip
\noindent Moreover, 
\begin{align*}
 \rank(F) &= | b \cap F|=| b \cap \bigcup\limits_{\alpha \subseteq F} \alpha| = \sum_{\alpha \subseteq F} |b \cap \alpha|\\
& = \sum_{\alpha \subseteq F} |b \cap (\bigcap \Gamma_{\alpha} -(\bigcap \Gamma_{\alpha} ~\cap~ \bigcup\limits \Gamma_{\alpha}^C))|\\
& = \sum_{\alpha \subseteq F} (|b \cap \bigcap \Gamma_{\alpha} | - | b \cap (\bigcap \Gamma_{\alpha} ~\cap~ \bigcup\limits \Gamma_{\alpha}^C)|)\\
& = \sum_{\alpha \subseteq F} (\rank(\bigcap \Gamma_{\alpha}) - \rank(\bigcap \Gamma_{\alpha} ~\cap~ \bigcup\limits \Gamma_{\alpha}^C))\\
& = \sum_{\alpha \subseteq F} \rank(M[~\bigcap \Gamma_{\alpha} ~\cap~ \bigcup\limits \Gamma_{\alpha}^C,\bigcap \Gamma_{\alpha}])\\
& = \sum_{\alpha \subseteq F} |b_{\alpha} \cap \alpha|  = \sum_{\alpha \subseteq F} |b' \cap \alpha| =|b' \cap F|.
\end{align*}
Since $|b' \cap F|=\rank(F)$ holds for all flacets $F$ of $\Gamma$, $b'$ is a basis of $M_{\omega}$, too. 
We finally have found a decomposition of $M_{\omega}$ in terms of its vertices. 
\end{proof}

\newpage
\begin{corollary}
Let $M$ be a connected matroid, $\Gamma=\{ F_1,\ldots,F_k  \}$ a set of flacets of $M$ and $\Pi:= \bigvee_{i}^k (F_i ~|~ E(M)-F_i) $. 
Then $\Gamma$ is the set of vertices of a face of $B(M)$ iff the following conditions are fulfilled: 
\begin{itemize}
\item \textnormal{For all $x$ in $E(M)$ if $x\in \alpha$ then there is at least one basis $b\in M$ s.t. \\
$|b \cap \bigcap \Gamma_{\alpha} |= \rank( \bigcap \Gamma_{\alpha})$, 
$|b \cap \bigcap \Gamma_{\alpha} \backslash \alpha|= \rank( \bigcap \Gamma_{\alpha} \backslash \alpha)$ and 
$x\in b$.} 
\item \textnormal{There is a basis $b\in M$ such that $|b \cap F |= \rank( F )$ for all $F \in \Gamma$ }
\item \textnormal{there is no other flacet $F'$ of $M$ with the property that $|b \cap F'|=\rank(F')$ for all bases $b \in M$, which fulfill 
$|b \cap F|=\rank(F)$ for all flacets $F$ of $\Gamma$.}
\end{itemize}
In this case  $\bigoplus_{\alpha \in \Pi} M[\bigcap \Gamma_{\alpha} \backslash \alpha,\bigcap \Gamma_{\alpha}]$ is 
the matroid type of the face and $\Gamma$ its set of flacets. 
\end{corollary}
\begin{proof}
This follows from Theorem \ref{theo1} and the part in the proof of theorem \ref{maintheo1} where it is shown that 
$\bigoplus_{\alpha \in \Pi} M[\bigcap \Gamma_{\alpha}\backslash \alpha,\bigcap \Gamma_{\alpha}~] ~\subseteq~ M$ if 
there is at least one basis 
satisfying $|b \cap F|=\rank(F)$ for all flacets $F$ of $\Gamma$. 
The second condition makes sure that the matroid type is indeed loopless. 
It may happen that the set $\Gamma$ determines a non-simplicial face of $B(M)$, for which there 
are still flacets which do not change the matroid type when adding their condition. The third point takes care of this. 
\end{proof}
\noindent For an example of the latter effect see \ref{ex5}. 
\bigskip
\begin{theorem}
\label{maintheo2}
Let $M$ be a connected matroid. 
For a face of its Bergman complex the decomposition of the matroid type in theorem \ref{maintheo1} is the finest one can get \textit{i.e.} 
the direct summands are connected. 
\end{theorem}
\noindent In order to show this we will prove some propositions and a lemma tailored just for this proof. 
\begin{lemma}
\label{lemma1}
 Let $A$ be a connected flat. We denote the connected components of $M[A,E(M)]$ by $K_1,\ldots,K_m$. Let $c$ be a circuit of $M$ and 
$x \in c \cap K_i $ for a certain $K_i$. Then there exists a circuit $c'$ of $M$ such that 
$c'- A \subseteq c - A$, $x \in c' \cap K_i  $ and $c'-A$ is a circuit of $M[A,E(M)]$.    
\end{lemma}
\begin{proof}
The idea of constructing such a circuit $c'$ from $c$ is the following. The problem with $c$ is that it can cut other connected 
components, too. 
Even $c_0 \cap K_i$ can be the disjoint union of several circuits of $M[A,E(M)]$. So we will use the strong elimination 
axiom \cite{Ox11} for matroids for cutting out parts of the circuit until we get a circuit of $M[A,E(M)]$.

 At first, we set $c_0:=c$. If $c_0$ had the property that $c_0 - A$ is a circuit of $M[A,E(M)]$, we would be done. So assume this is not 
the case. 
Then there exists another circuit $\tilde{c}_0$ of $M$ such that $ \tilde{c_0} - A \subsetneq c_0 -A $. 
Note that in this case $c_0 \cap A \neq \emptyset$ is a necessary condition. \\
Therefore $\tilde{c_0} \cap A \neq \emptyset$, since otherwise $\tilde{c_0} = \tilde{c_0} - A \subsetneq c_0 -A \subsetneq c_0$, 
which is a contradiction to $c_0$ being a circuit. \\
If $x \in \tilde{c_0}$, again, we would be done. Otherwise, we choose any $y_0 \in \tilde{c_0} -A$. 
The strong elimination axiom for matroids grants us the existence of a new circuit $c_1$ with the properties that 
$c_1 \subseteq (c_0 \cup \tilde{c_0}) - \{y_0\}$ and $x \in c_1$. \\
Let us compare $c_0$ and $c_1$. Like $c_0$ our new circuit contains $x$. Constructing $c_2, c_3, \ldots $ , we can repeat this process, 
as long as 
we have to. 
\begin{displaymath} c_{i+1} - A \subseteq (\tilde{c_i} \cup c_i) - \{ y_i \} - A = ((\tilde{c_i}-A) \cup (c_i -A)) - \{ y_i \}  =  c_i -A - \{y_i\} 
\end{displaymath}
Since $y_i \in \tilde{c_i} -A \subsetneq c_i - A$ , we remove at least one element by this process. 
But there are only finitely many elements we can remove, so the construction has to end with either $c_i -A$ or $\tilde{c_i} -A$ satisfying 
the condition of being a circuit in $M[A,E(M)]$ for some $i$.
\end{proof}
\begin{proposition}
\label{prop3}
Let $A$ be a connected flat of $M$ and $K_1,\ldots,K_t$ connected components of $M[A,E(M)]$. 
Then $F_i := A \cup \bigcup_{j \neq i} K_j = E(M)-K_i$ is a flacet of~$M$. 
\end{proposition}
\begin{proof}
 We want to show the following properties:
\begin{enumerate}
\setstretch{0.1}
\item $F$ is connected \\
\item $F$ is co-connected \\
\item $F$ is a flat 
\end{enumerate}
For the first condition, we start with a circuit $c$ which connects a given element $x \in F_i - A $ with any element $y \in A$. Its existence is 
guaranteed by the connectedness of $M$. Our previous lemma grants us a circuits $c' $ with $c' \subseteq F_i$, $c' \cap A \neq \emptyset$ and  
$x \in c'$. So every $x \in F_i -A$ is connected to an element of $A$ by a circuit which is disjoint to $K_i$. 
Since $A$ is connected itself transitivity of connectedness yields that $M[\emptyset,F_i]$ is connected, too.

\bigskip
\noindent For the second part, 
\begin{align*}
 M[F_i,E(M)] & =  M[A \cup \bigcup\limits_{j \neq i} K_j , E(M) ]\\
& \cong \left( \bigoplus\limits_{i=1}^t M[A,A \cup K_j]\right) ~[~\bigcup\limits_{j \neq i} K_j,
\underbrace{\bigcup\limits K_j}_{E(M)-A}] \\
& \cong 
M[A,A\cup K_i]  .
\end{align*}
So $F_i$ is co-connected since $K_i$ was a connected component. 

Now the third part is less work. Since $M[F_i,E(M)]$ is even connected, this matroid is loopless in particular. Thus $F_i$ is a flat.
\end{proof}
\begin{proposition}
\label{prop4}
Let $A \subseteq E(M)$ and $A_1,\ldots,A_n$ connected components of $M[\emptyset,A]$. 
We denote the connected components of $M[A,E(M)]$ by $K_1,\ldots,K_t$. If $A$ has full $\omega$-rank, then for any $K_i$ and any connected component 
$G$ of $M[\emptyset,A~\cup~K_i]$, which intersects $K_i$, both $A \cup K_i$ and $G$ have full $\omega$-rank, too.
\end{proposition}
\begin{proof}
Our goal is to show that $|b\cap (G)|= \rank(G)$ holds for all bases $b$ of $M_{\omega}$. 
For any such basis the equation $ | A \cap b| = \rank(A)$ holds because $A$ has full $\omega$-rank. 
By Proposition \ref{prop2} all the $A_i$ and any union of them have full $\omega$-rank, too.     
Hence, \begin{displaymath} |b \cap K_i| + |b \cap A| = |b \cap (K_i \cup A)| \leq \rank(K_i \cup A),\end{displaymath}
which implies \begin{displaymath} |b \cap K_i| \leq \rank(K_i \cup A) - |b \cap A| =  \rank(K_i \cup A) - \rank(A).\end{displaymath}
On the other hand $b-A$ is a basis of $M[A,E(M)]$ and thus: 
\begin{align*}
 \sum_{i=1}^t |b \cap K_i| & = | b \cap \bigcup\limits_{i=1}^t K_i| = | b -A|= 
\rank(M[A,E(M)])\\
 & = \rank\left(\bigoplus\limits_{i=1}^t M[A,A \cup K_i]\right)
  = \sum_{i=1}^t \rank(M[A,A \cup K_i]) \\ 
 & = \sum_{i=1}^t (\rank(M[\emptyset,A \cup K_i])-\rank (M[\emptyset, A]))\\
 & = \sum_{i=1}^t (\rank(A \cup K_i)-\rank(A)).
\end{align*}
In a nutshell, $|b \cap K_i| \leq \rank(K_i \cup A) - \rank(A)$ and summing over all $K_i$ we obtain equality, so: 
\begin{align*}
&& |b \cap K_i| &= \rank(K_i \cup A) - \rank(A)\\
 \Rightarrow  
&& |b \cap K_i| +\rank(A) &= \rank(K_i \cup A) \\
 \Rightarrow 
&& |b \cap K_i|+ |b \cap A|  &= \rank(K_i \cup A) \\
 \Rightarrow 
&& |b \cap (K_i \cup A)| &= \rank(K_i \cup A).
\end{align*}

\noindent Again, by Proposition \ref{prop2}, all the connected components of 
$K_i \cup A$, such as our $G$, have full $\omega$-rank, too. 
\end{proof}
\begin{proposition}
\label{prop5}
 Every connected flat $A$ of a matroid $M$ is an intersection of flacets of $M$. Furthermore, if $A$ has full $\omega$-rank, 
then so do these flacets. This means they are vertices of the face $M_{\omega}$ of the Bergman complex. 
\end{proposition}
\begin{proof}
 Let $A$ be a connected flat of $M$ and $K_1,\ldots,K_t$ the connected components of $M[A,E(M)]$. Due to Proposition \ref{prop3} the sets 
$A \cup \bigcup_{j \neq i} K_j$ are flacets for all~$i$, and
\begin{displaymath}A=\bigcap_{i=1}^t ~\left(A \cup \bigcup\limits_{j \neq i} K_j\right) .\end{displaymath}
For the second part, consider Proposition \ref{prop4}. 
It states that $A \cup K_i$ has full $\omega$-rank. And so by Proposition \ref{prop2} 
$\bigcup_{j \neq i} (A \cup K_j)~ = ~A \cup \bigcup_{j \neq i} K_j$ has full $\omega$-rank. 
\end{proof}

\noindent Now we can finally prove theorem \ref{maintheo2}. 
\begin{proof}
 First we want to show that the matroid $M[\emptyset,\bigcap \Gamma_{\alpha}]$ is connected. Let $A_1,\ldots,A_n$ be the connected components of
$M[\emptyset,\bigcap \Gamma_{\alpha}]$. The flat $\bigcap \Gamma_{\alpha}$ has full $\omega$-rank as an intersection of 
flacets with full $\omega$-rank. 
The same is true for its individual connected components by Proposition \ref{prop2}.  
So due to Proposition \ref{prop5} there exists $\Gamma_i \subseteq \Gamma$ for all $i$ such that $A_i=\bigcap \Gamma_i$. 
We know that $\emptyset~\neq~\alpha~\cap~A_i~=~\alpha \cap (\bigcap \Gamma_i)$. 
This implies that $\alpha \cap F \neq \emptyset$ for all $F \in \Gamma_i $, 
which is indeed equivalent to $\alpha \subseteq F $ for all $F \in \Gamma_i $. 
Therefore we know that $\alpha \subseteq \bigcap \Gamma_i = A_i$. 

Since the $A_i$ are pairwise disjoint, $\alpha$ must be contained in exactly one of them. 
Let us say this connected component is $A_i$ and $A_j$ is any other component disjoint to $\alpha$. Then,  
\begin{displaymath}\alpha \subseteq A_i = \bigcap \Gamma_i ~\Rightarrow ~ \Gamma_i \subseteq \Gamma_{\alpha} ~ \Rightarrow~\emptyset = A_j \cap A_i =
 A_j \cap \bigcap \Gamma_i \supseteq A_j \cap \bigcap \Gamma_{\alpha} = A_j .\end{displaymath}
So there must not exist any other connected component than $A_i$. Thus $M[\emptyset,\bigcap \Gamma_{\alpha}]$ is connected. 

\bigskip
\noindent Now we will see why we put all this work in the previous propositions. 
The requirements of Proposition \ref{prop4} are satisfied for 
\begin{itemize}
\item the matroid $M[\emptyset,\bigcap \Gamma_{\alpha}]$, 
\item the set $\bigcap \Gamma_{\alpha} ~\cap~ \bigcup\limits \Gamma_{\alpha}^C$, which equals 
$\bigcap \Gamma_{\alpha} \backslash \alpha $, as $A$, 
\item any connected component $K_i\subseteq \alpha$ of 
$M[\bigcap \Gamma_{\alpha}\backslash \alpha ,\bigcap \Gamma_{\alpha}]$,  
\item any
connected component $G$ of $M[\emptyset,(\bigcap \Gamma_{\alpha}\backslash \alpha)\cup K_i]$ intersecting 
$K_i~\subseteq~\alpha$. 
\end{itemize}
Proposition \ref{prop4} states that $G$ has full $\omega$-rank, too. In particular $G$ is a flat, due to Proposition \ref{prop1}, which is connected by 
construction as a connected component. So again with Proposition \ref{prop5} we obtain $G=\bigcap \Gamma'$ for some $\Gamma' \subseteq \Gamma$. 
We know that $\emptyset \neq K_i \cap G \subseteq \alpha \cap G= \alpha \cap (\bigcap \Gamma')$. 
This implies that $\alpha \cap F \neq \emptyset $ for all flacets $F \in \Gamma' $, 
which is already equivalent to $\alpha \subseteq F$ for all flacets $F \in \Gamma' $. 
Therefore we know that $\Gamma' \subseteq \Gamma_{\alpha}$, which yields to $\bigcap \Gamma_{\alpha} \subseteq \bigcap \Gamma'$.  

\noindent So we can conclude about $G$: 
\begin{displaymath}\bigcap \Gamma_{\alpha} ~\subseteq~ \bigcap \Gamma' ~= ~G ~ \subseteq ~ (\bigcap \Gamma_{\alpha} ~\cap~ \bigcup\limits \Gamma_{\alpha}^C)\cup K_i ~ 
\subseteq ~\bigcap \Gamma_{\alpha}.\end{displaymath}
Thus $G$ equals $\bigcap \Gamma_{\alpha}$ and $K_i$ equals $\alpha$. Therefore 
$M[\bigcap \Gamma_{\alpha}\backslash \alpha, \bigcap \Gamma_{\alpha}]$ has only one connected component. 
\end{proof}

\begin{corollary}
The decomposition of the matroid type in \ref{maintheo1} 
for a face of the order complex is coarser than the decomposition of the matroid type of its supporting face in the 
nested set complex. 

The decomposition of the matroid type for a face of the nested set complex is coarser than the 
decomposition of the matroid type of its supporting face in the Bergman complex.
\end{corollary}
\begin{proof}
 We just have to compare the partitions $\bigvee_{F \in \Gamma} (F | E(M)- F)$ for the different sets of vertices. 
Feichtner and Müller \cite{FM05} showed that the nested set which is the support of the chain consists of the connected components of 
the chain elements. So the partition is becoming finer when passing to more vertices. 
Due to \ref{maintheo2} the direct summands for faces of the Bergman complex are connected, so this must be the finest decomposition. 
\end{proof}
Note that these relations are strict iff the dimension of the faces is increasing while taking the support face in the next coarser complex. 
So for maximal faces the decompositions are all the same. Using this we can easily see that the maximal faces of the Bergman complex correspond 
to transversal matroids \textit{i.e.} direct sums of matroids of rank 1.

\section{Taking connected components}
\label{sec6}
\label{ausblick}
For the following the term sublattice is meant to be including the elements $\hat{0}$ and $\hat{1}$. 

Last but not least i want to share some thoughts about further research in this direction. 
Remember the different levels of subdivisions of the Bergman complex. 
The finest is the order complex, then there are the nested set complexes respective various building sets and the coarsest is the Bergman complex itself. 
For a chain $\emptyset =c_0<\ldots c_k<c_{k+1}=E(M)$ in the geometric lattice the smallest sublattice containing this chain 
and no additional requirements $\mathcal{L}_C$ is the chain itself. 

With each nested set we can identify the smallest sublattice $\mathcal{L}_S$ containing $S$ and fulfilling the property: 
If $A\in \mathcal{L}$ then all connected components of the interval $[\emptyset,A]$ are in $\mathcal{L} $ as well. 
This special sublattice $\mathcal{L}_S$ is exactly the set of flats whose connected components are elements of $S$. 

Now for a matroid type $\Mw$ consider the set of flats with $\omega$-full rank $\mathcal{L}_{\omega}$. 
If $\Gamma$ are the vertices of the face of the Bergman complex corresponding to this matroid type, 
$\Gamma$ are exactly the flacets of $M$ contained in $\mathcal{L}_{\omega} $. 
Additionally we can see $\mathcal{L}_{\omega} $ as the smallest sublattice containing $\Gamma$ and fulfilling the property:
If $A<B\in \mathcal{L}$ then all connected components of the interval $[A,B]$ are in $\mathcal{L} $ as well. 

So all the faces of the different levels of subdivisions can be seen as induced sublattices with certain properties 
which can be described purely 
in terms of lattice theory. 
This interpretation even fits together with the subdivisions. 
Starting with a chain $C$ the nested set, which is the support of this chain, consists of the connected flats, 
which are contained in the smallest 
sublattice containing $C$ and being closed under taking connected components (with left interval border~$\emptyset$) like above. 

The same way starting with a nested set $S$ the corresponding sublattice of the supporting matroid type is the smallest sublattice 
containing $S$ which is closed under taking any connected components.  

Now these descriptions of the different faces as special sublattices are made just in terms of order theory. 
They can of course be generalized to arbitrary lattices (no geometric lattices anymore). 
The question would be if this "Bergman complex" is still homotopy equivalent to the order complex. 
For me it feels like this could be true, but there are two major problems left yet. 
On the one hand not every sublattice closed under taking connected components is indeed a face of the Bergman complex and 
I have not found 
an acceptable condition for them to be.  
On the other hand in some examples the reduction of faces works "too" good. 
The question would be what kind of complex the result actually is. 
There are pathological examples where the order complex is just a subdivided circle and 
the Bergman complex consists of just one 1-dimensional face 
and no vertices at all.  

\newpage
\bibliography{BergmanReview}{}

\providecommand{\bysame}{\leavevmode\hbox to3em{\hrulefill}\thinspace}
\providecommand{\MR}{\relax\ifhmode\unskip\space\fi MR }
\providecommand{\MRhref}[2]{%
  \href{http://www.ams.org/mathscinet-getitem?mr=#1}{#2}
}
\providecommand{\href}[2]{#2}
\begin{thebibliography}{GGMS87}

\bibitem[AK05]{AK}
F.~Ardila and C.J. Klivans, \emph{{The Bergman complex of a matroid and
  phylogenetic trees}}, J. Combin. Theory Ser. B \textbf{96} (2005), no.~1,
  38--49.

\bibitem[Ber71]{BE71}
G.M. Bergman, \emph{{The logarithmic limit-set of an algebraic variety}},
  Transactions of the American Mathematical Society (1971), 459--469.

\bibitem[FK04]{FK04}
E.M. Feichtner and D.N. Kozlov, \emph{{Incidence combinatorics of
  resolutions}}, Selecta Math. (N.S.) \textbf{10} (2004), no.~1, 37--60.

\bibitem[FM05]{FM05}
E.M. Feichtner and I.~M{\"u}ller, \emph{{On the topology of nested set
  complexes}}, Proc. Amer. Math. Soc. \textbf{133} (2005), no.~4, 999--1006.

\bibitem[FS05]{FS04}
E.M. Feichtner and B.~Sturmfels, \emph{{Matroid polytopes, nested sets and
  Bergman fans}}, Port. Math. (N.S.) \textbf{52} (2005), no.~4, 437--468.

\bibitem[FY04]{FY04}
Eva~Maria Feichtner and Sergey Yuzvinsky, \emph{{Chow rings of toric varieties
  defined by atomic lattices}}, Inventiones Mathematicae \textbf{155} (2004),
  515--536.

\bibitem[GGMS87]{GelGor}
I.M. Gelfand, R.M. Goresky, R.D. MacPherson, and V.V. Serganova,
  \emph{{Combinatorial geometries, convex polyhedra, and Schubert cells}}, Adv.
  Math \textbf{63} (1987), no.~3, 301--316.

\bibitem[Oxl11]{Ox11}
J.~Oxley, \emph{{Matroid theory}}, Oxford University Press, 2011.

\bibitem[Stu02]{STU02}
B.~Sturmfels, \emph{{Solving systems of polynomial equations}}, Amer.Math.Soc.,
  CBMS Regional Conferences Series, No 97, Providence, Rhode Island, 2002.

\bibitem[{Whi}92]{edWhite}
N.~{White (ed.)}, \emph{{Matroid applications.}}, {Encyclopedia of Mathematics
  and Its Applications. 40. Cambridge: Cambridge University Press. xii, 363 p.
  }, 1992.

\end{thebibliography}
\bibliographystyle{amsalpha}
\end{document}